\documentclass[11pt]{amsart}

\theoremstyle{definition}
\newtheorem{definition}{Definition}
\theoremstyle{plain}
\newtheorem{theorem}{Theorem}

\newtheorem{lem}{Lemma}
\newtheorem{prop}{Proposition}

\theoremstyle{remark}
\newtheorem{remark}{Remark}
\DeclareMathOperator{\RR}{\mathbb R}

\DeclareMathOperator{\NN}{\mathbb N}

\begin{document}

\title{Asymptotic Smoothness, Convex Envelopes and Polynomial Norms}

\author[R. Gonzalo]{Raquel Gonzalo}
\address{R. Gonzalo: Departamento de Matem\'atica Aplicada, Facultad de
Inform\'atica, Universidad Polit\'ec\-ni\-ca, Campus de
Montegancedo, Boadilla del Monte, 28660-Madrid (Spain).}
\email{rngonzalo@fi.upm.es}

\author[J. A. Jaramillo]{Jes\'us Angel Jaramillo}
\address{ J. A. Jaramillo: Instituto de Matem\'atica Interdiscliplinar (IMI), Departamento
de An\'alisis Matem\'atico, Universidad Complutense de Madrid, 28040-Madrid (Spain).}
\email{jaramil@mat.ucm.es}

\author[D. Y\'a\~{n}ez]{Diego Y\'a\~{n}ez}
\address{D. Y\'a\~{n}ez: Departamento de Matem\'aticas,
Escuela de Ingenierías Industriales,
Universidad de Extremadura, 06006-Badajoz (Spain).}
\email{dyanez@unex.es}

\thanks{Partially supported by MINECO (Spain) Project MTM2012-34341}

\date{}

\begin{abstract}

We introduce a suitable notion of {\it asymptotic smoothness} on infinite dimensional Banach spaces, and  we prove that, under some structural restrictions on the space, the convex envelope of an asymptotically smooth function is asymptotically smooth. Furthermore, we study convexity and smoothness properties of polynomial norms, and we obtain that a polynomial norm of degree $N$ has modulus of convexity of power type $N$.
\end{abstract}

\maketitle

\section{Introduction}

Let  $X$ be a Banach space and consider a function $f:X\to \mathbb{R}$, which we assume to be continuous and bounded below. Recall that the {\it convex envelope} of $f$ is defined as the greatest convex function majorized by $f$, that is

$$ \mbox{\sl conv} f (x)=\sup\{ g(x) \, / \,  g:X\to \mathbb{R}, \,  g \mbox{ convex }, \, g\leq f \}.$$

It is no difficult to see that

$$  \mbox{\sl conv} f (x)=\inf{\Big{\{}}\sum_{i=1}^{k} \lambda_{i} f(x_{i})\, / \, \lambda_{i}\geq 0, \, \sum_{i=1}^{k} \lambda_{i}=1, \, \sum_{i=1}^{k}\lambda_{i} x_{i}=x, \, k \in \mathbb N {\Big{\}}}.$$

When $X$ is  a finite dimensional space of dimension $n$, then by a classical result of Carath\'eodory  we can restrict the above sums to the case $k=n+1$.

\

In this note we are interested in how {\sl conv}$f$ inherits smoothness properties from $f$.  Not much seems to be known about this in the case of general Banach spaces. In the finite-dimensional case, this usually requires some restrictions on the growth of the function $f$. For example, it was proved by Griewank and Rabier  \cite{GR} that, if $f:\mathbb R^n \to \mathbb R$ is of class $C^1$ (respectively, $C^{1, \alpha}$), then so is {\sl conv}$f$, provided
$$
\lim_{\|x\|\to\infty}\dfrac{f(x)}{\|x\|}=\infty.
$$
On the other hand, Kirchheim and Kristensen \cite{KK} have obtained that  {\sl conv}$f$ is differentiable when $f$ is, provided
$$
\lim_{\|x\|\to\infty}f(x)=\infty.
$$
But is is well known that, in general, {\sl conv}$f$  does not share the differentiability properties of  $f$. For example (see \cite{BH}) the function $f(x,y)=\sqrt{x^2+e^{-y^2}}$ is everywhere differentiable on $\mathbb R^2$ and nevertheless  {\sl conv}$f$ is not. Concerning high order smoothness, it is easy to construct examples of polynomial functions on the real line whose convex envelope is not even  $C^2$ smooth.

\

Our aim here is to show that, in some sense, this loosing of differentiability is essentially a finite-dimensional phenomenon. More precisely, we will introduce a suitable notion of {\it asymptotic smoothness} on infinite-dimensional Banach spaces, and then we will prove that, under some structural restrictions on the space, the convex envelope of an asymptotically smooth function will be again asymptotically smooth. Furthermore, we will also study convexity and smoothness properties of polynomial norms in Banach spaces.
\

The contents of the paper are as follows. In section 2 we review the basic notions of uniform convexity and smoothness of a Banach spaces, as well as their asymptotic counterparts. In section 3, we introduce the related notion of asymptotically $p$-smooth function, for $1<p< \infty$. We also prove that asymptotic $p$-smoothness can be characterized by a suitable modulus, in spaces where polynomials of degree $<p$ have nice weak continuity properties. In section 4 we obtain that, in this class of spaces, the convex envelope of an asymptotically $p$-smooth function is asymptotically $p$-smooth. Finally, section 5 is devoted to study the moduli of convexity and smoothness of polynomial norms in Banach spaces. As a main result here, we prove that a polynomial norm of degree $N$ has modulus of convexity of power type $N$.

\section{Preliminaires}

We start this section by recalling the classical notions of uniform convexity and smoothness for a Banach space $(X, \Vert \cdot \Vert)$. The {\it modulus of uniform convexity} of $X$ is defined for $\varepsilon \in (0,2]$ as
$$
\delta_X(\varepsilon) = \inf \left\{ 1 - \frac{\Vert x +y \Vert}{2}:
x,y\in X; \Vert x \Vert = \Vert y \Vert=1, \Vert x- y \Vert
=\varepsilon \right\}.
$$
The space $X$ is said to be {\it uniformly convex} if $\delta_X(\varepsilon)>0$ for all $\varepsilon \in (0,2]$. Furthermore,
$X$ is said to have modulus of convexity of power type $p$ if there is a constant $C>0$ such that $\delta_X(\varepsilon) \geq C
\varepsilon^p$ for for all $\varepsilon \in (0,2]$. On the other hand, the {\it modulus of uniform
smoothness} of $X$ is defined for $\tau \in (0,2]$ as
$$
\rho_X(\tau) = \sup \left\{ \frac{\Vert x +\tau y \Vert + \Vert x -
\tau y \Vert}{2}-1: x,y\in X; \Vert x \Vert = \Vert y \Vert=1,
\Vert x- y \Vert =\tau \right\}.
$$
The space $X$ is said to be {\it uniformly smooth} if
$\lim_{\tau \to 0} \frac{\rho_X(\tau)}{t}=0$. Furthermore, $X$ is
said to have modulus of smoothness of power type $p$ if there is a
constant $C>0$ such that $\rho_X(\tau) \leq C \tau^p$ for for all
$\tau \in (0,2]$.

\

We now recall the concepts of asymptotic uniform convexity and smoothness, and the corresponding moduli. These notions  were first introduced by
Milman \cite{Milman} under different notations and names. We refer to \cite{JLPS} for a  survey about the most relevant results concerning these moduli. Along the paper we will use the same notation and terminology as in \cite{JLPS}.

\

For a  Banach space $(X, \Vert \cdot \Vert)$, the {\it
pointwise modulus of asymptotic convexity} is defined for each $\Vert x
\Vert =1$ and $t>0$ by:
$$
\overline{\delta}(t;x)= \sup_{dim(X/H)<\infty} \qquad \inf_{h\in H,
\,\, \Vert h \Vert \geq t} \qquad \Vert x+h \Vert -1,
$$
\noindent and the {\it modulus of asymptotic uniform convexity} is
defined for  $t>0$ by:
$$
\overline{\delta}_X(t)= \inf_{\Vert x \Vert =1} \delta(t;x).
$$
\noindent The {\it pointwise modulus of asymptotic smoothness } is
defined for $\Vert x \Vert =1$ and $t>0$ by:
$$
\overline{\rho}(t;x)=\inf_{dim(X/H)<\infty} \qquad \sup_{h\in H,
\,\, \Vert h \Vert \leq t}  \qquad\Vert x+h \Vert -1,
$$
\noindent and the {\it modulus of asymptotic uniform smoothness }
is defined for $t>0$ by:
$$
\overline{\rho}_X(t)= \sup_{\Vert x \Vert =1} \rho(t;x).
$$

\noindent The space $X$ is said to be {\it asymptotically uniformly
convex} if $\overline{\delta}(t)>0$ for every $t>0$, and {\it
asymptotically uniformly smooth} if $\overline{\rho}(t)/t \to 0$ as
$t\to 0$.  Usually, one is interested in  power-type estimates of
moduli. More precisely,  for $1\leq p<\infty$ we say that $X$
has modulus of {\it asymptotic convexity of power type} $p$  if
there exists $C>0$ such that $\overline{\delta}(t) \geq Ct^ p$, for
all $0< t \leq 1$. In the same way, $X$ has modulus of {\it
asymptotic smoothness of power type} $p$
if there exists $C>0$ so that $\overline{\rho}(t)
\leq Ct^p$, for all $0 <t \leq 1$. For example, in the case of
space $\ell_p$ with $1\leq p<\infty$, we have that
$\overline{\rho}(t)=\overline{\delta}(t)=(1+t^p)^{1/p}-1$; and for
the space $c_0$ we have that
$\overline{\rho}(t)=\overline{\delta}(t)=0$ for $0<t\leq 1$. As a consequence,
$\ell_p$ has modulus of  asymptotic convexity and smoothness of
power type $p$ for every $1\leq p<\infty$. Note in particular that
$\ell_1$ is  asymptotically uniformly convex and $c_0$ is
asymptotically uniformly smooth.

\

Now if, instead of the norm, we consider a continuous convex function
$f: X \to \mathbb R$, it is possible to define the analogous notions for $f$ along the same lines.
The {\it pointwise modulus of asymptotic  convexity} of $f$ is defined for  $x\in X$ and  $t>0$ by:
$$
\overline{\delta}_{f}(t;x) = \sup_{dim(X/H)<\infty} \, \,
\inf_{h\in H, \,\, \Vert h \Vert \geq t}  \, \, f(x+h) -f(x).
$$
\noindent Similarly, the {\it pointwise modulus of asymptotic smoothness} of $f$ is defined for  $x\in X$
and  $t>0$ by:
$$
\overline{\rho}_{f}(t;x)=\inf_{dim(X/H)<\infty} \, \,
\sup_{h\in H, \, \, \Vert h \Vert \leq t} \, \, f(x+h) -f(x).
$$
\noindent Finally, the  {\it moduli of asymptotic uniform convexity and  smoothness}  of the function $f$
on a given subset $S$ of $X$ are defined for $t>0$ by:
$$
\displaystyle{\overline{\delta}_{f}(t;S)= \inf_{x\in S} \, \,
\sup_{dim(X/H)<\infty} \, \, \inf_{h\in H, \Vert y \Vert \geq t} f(x+h)-f(x).}
$$
$$
\displaystyle{\overline{\rho}_{f}(t;S)= \sup_{x \in S} \, \,
\inf_{dim(X/H)<\infty} \, \, \sup_{h\in H, \Vert h \Vert \leq t}
f(x+h)-f(x).}
$$

\noindent It is  clear that when  $f(x)= \Vert \cdot \Vert$ and $S=S_{X}$ is the unit sphere of the Banach space $(X,\Vert \cdot \Vert)$ we obtain the corresponding definitions of the moduli of asymptotic uniform convexity and smoothness of the space.

\

The following Lemma provides an alternative description of the pointwise modulus of asymptotic smoothness of a continuous convex function.

\

\begin{lem} \label{absolute}
Let $X$ be a Banach space and $f:X \to \RR$ a continuous convex function. Then, for every $x\in X$ and $t>0$ we have that
$$
\overline{\rho}_{f}(t;x)=  \inf_{dim(X/H)<\infty} \, \, \sup_{h\in H, \, \, \Vert h \Vert \leq t} \, \, \vert f(x+h) -f(x) \vert .
$$
\end{lem}

\begin{proof} Fix  $x\in X$ and $t>0$, and denote
$$
\alpha = \inf_{dim(X/H)<\infty} \, \, \sup_{h\in H, \, \, \Vert h \Vert \leq t} \, \, \vert f(x+h) -f(x) \vert .
$$
Since $f(x+h)-f(x) \leq \vert  f(x+h)-f(x) \vert$ we always have that $\overline{\rho}_{f}(t;x) \leq \alpha$. Assume that the inequality is strict. Then there exists a finite codimensional subspace $H$ of $X$ such that
$$
\sup_{h\in H, \, \, \Vert h \Vert \leq t} f(x+h) -f(x) < \alpha.
$$
Now since $f$ is a continuous convex function, it is well known (see e.g. \cite{libro}) that the subdifferential of $f$ at  $x$,
denoted by $\partial{f}(x)$, is a non-empty subset of $X^{*}$. Then there exists some  $x^{*} \in \partial{f}(x)$. This means that, for every $h\in X$,
$$
f(x+h) -f(x) \geq x^{*}(h).
$$
We now may consider  the finite codimensional subspace $H^{*}= H \cap ker(x^{*})$ and then:
$$
\alpha \leq \sup_{h\in H^{*}, \, \,  \Vert h \Vert \leq t} \vert f(x+h) -f(x)  \vert = \sup_{h\in H^{*}, \, \,  \Vert h \Vert \leq t}
 f(x+h) -f(x)  < \alpha,
 $$
which is a contradiction.
\end{proof}

\begin{remark} Let $X$ be a Banach space and $f:X \to \RR$ a continuous convex function. It follows from Lemma \ref{absolute} that for each $x\in X$ and and $0<t\leq 1$ we have that $\overline{\rho}_{f}(t;x) \geq 0$. A similar reasoning shows that also $\overline{\delta}(t;x) \geq 0$.
\end{remark}

\section{High order asymptotic smoothness and flatness. }

In \cite{GJT} several notions of high order smoothness, defined in terms of Taylor polynomial expansions, are considered. In particular,  if $X$ is a Banach space and  $1<p<\infty$, a function  $f:X \to \mathbb R$ is said to be {\it $U^p$-smooth} on a subset $S$ of $X$ if for each  $x\in S$
there exists a polynomial  $P_x$ on $X$ of degree  $< p$, with  $P(0)=0$, verifying that
$$|f(x+h)-f(x)-P_x(h)|= O (\|h\|^p),$$
\noindent uniformly on  $x\in S$. Here, as usual, by a polynomial of degree $k$ on $X$ we mean a function $P:X \to \mathbb R$ of the form $P= P_0+P_1+ \cdots +P_k$, where $P_0$ is constant and, for each $j= 1, \dots , k$, $P_j(x)$ is a $j$-homogeneous polynomial. This means that there exists a symmetric, $j$-linear, continuous mapping $A_j: X \times \cdots \times X \to \mathbb R$ such that $P_j(x) = A_j (x, \cdots , x)$ for all $x \in X$. For example, suppose that $k$ is the greatest integer strictly less than $p$. If $f$ is $C^k$-smooth and its $k$-th derivative is uniformly $(p- k)$-H\"{o}lder continuous on $X$, an application of Taylor formula with integral remainder gives us that $f$ is $U^p$-smooth on $X$.

\

We now introduce the notion of asymptotic $p$-smoothness which will be useful for our purposes as an ``asymptotic" version of $U^p$-smoothness:

\begin{definition}  Let $X$ be a Banach space and  $1<p<\infty$. We say that a function  $f:X \to \mathbb R$ is  {\it asymptotically $p$-smooth} at $x\in X$ if there exist $K_x>0$,  $\delta_x>0$ and a polynomial $P_x$ of degree  $< p$, with  $P(0)=0$, such that for each  $0<t<\delta_x$ there is a finite codimensional subspace $H_{(t,x)}$ of $X$,  verifying that if $h\in H_{(t,x)}$  and $\Vert h \Vert \leq t$:
$$
\vert f(x+h) -f(x) -P_x(h)\vert \leq K_{x}\, \|h\|^p.
$$
If  the constants $K_x$ and $\delta_x$ are uniform for all $x$ in a a subset $S$ of $X$, we say that $f$ is {\it uniformly asymptotically $p$-smooth on $S$}.
\end{definition}

As we will see, the concept of asymptotic smoothness is related to the following modulus:

\begin{definition} Let $X$ be a Banach space and consider  a function $f:X \to \mathbb R$. For each $t>0$, we define the {\it modulus of asymptotic smoothness} of $f$  at a point $x\in X$ as:
$$
\overline{\rho}_{f}(t;x):= \inf_{dim(X/H)<\infty} \sup_{h\in H, \Vert h \Vert \leq t}
\vert f(x+h) -f(x) \vert \in [0, \infty].
$$
In the same way, we define the {\it modulus of uniform asymptotic smoothness} of $f$  on a subset $S$ of $X$ as:
$$
\overline{\rho}_{f}(t;S):
= \sup_{x \in S} \overline{\rho}_{f}(t,x)= \inf_{x\in S} \inf_{dim(X/H)<\infty} \sup_{h\in H, \Vert h \Vert \leq t}
\vert f(x+y) -f(x) \vert  \in [0, \infty].
$$
We say that $f$ is {\it asymptotically flat} at $x$ (resp. {\it uniformly asymptotically flat} on $S$) if there exists $\tau>0$ such that $\overline{\rho}_{f}(t;x) =0$ (resp. $\overline{\rho}_{f}(t;S)=0$) for all $0<t<\tau$.
\end{definition}

Note that, from Lemma \ref{absolute}, this definition coincides with the one given in the previous Section for convex functions. On the other hand, it is cleat that if a function is asymptotically flat at a point (resp. uniformly on a subset $S$) then it is asymptotically $p$-smooth for every $p>1$ (resp. uniformly on $S$). Note in particular that the usual sup-norm $f(x)= \Vert x\Vert_{\infty}$ on the space $c_0$ is an example of non-differentiable function which is asymptotically $p$-smooth for every $p>1$.

\

We are going to show that, in some spaces, asymptotic $p$-smoothness can be characterized in terms of the above modulus. This requires some weak continuity properties of polynomials on the space. Recall that a function $f:X \to \mathbb R$ defined on a Banach space $X$ is said to be {\it wb-continuous} if, for every bounded subset $B$ of $X$ and every $\varepsilon>0$, there exist $\delta >0$ and functionals $x^*_1, \dots , x^*_m \in X^*$ such that $\vert f(x) - f(y) \vert < \varepsilon$ whenever $ x, y \in B$ with $\vert x^*_j (x-y) \vert <\delta$ for every $j=1, \dots , m$. It follows from \cite{FGL} that, if $X$ contains no copy of $\ell_1$, {\it wb}-continuity coincides with  sequential continuity for the weak topology of $X$. We refer to \cite {GJT} and \cite{DGJ} for information about the connections between smoothness properties of the space $X$ and {\it wb}-continuity of polynomials up to a given degree. For example, it follows from Corollary 2.2 in \cite{DGJ} that, if $X$ contains no copy of $\ell_1$ and $\overline{\rho}_X(t)= o(t^k)$, then  every polynomial on $X$ of degree $\leq k$ is {\it wb}-continuous. We need the following simple Lemma, where, for a subspace $Y$ of $X$, we denote by
$$\|f\|_{Y}= \sup\{|f(y)|: \Vert y \Vert = 1, y\in Y\}.$$

\begin{lem}\label{wb}
Let $X$ be a Banach space and let $f:X  \to \mathbb{R}$ be a {\it wb}-continuous function with $f(0)=0$. Then, for every
$\varepsilon >0$ there exists a finite codimensional subspace  $Y$ of $X$ such that $\|f\|_{Y}<\varepsilon$.
\end{lem}
\begin{proof} Given $\varepsilon>0$ there are $\delta >0$ and functionals $x^*_1, \dots , x^*_m \in X^*$ such that if $x$ belongs to the unit ball $B$ of $X$ and $\vert x^*_j (x) \vert <\delta$ for every $j=1, \dots , m$, then $\vert f(x)  \vert < \varepsilon$. Now if we consider the finite codimensional subspace $Y =\cap_{j=1}^m ker (x^*_j)$ we have that  $\|f\|_{Y}<\varepsilon$.
\end{proof}

\begin{prop} \label{flat}
Let $1<p <\infty$, let $X$ be a Banach space such that all polynomials of degree $< p$ are {\it wb}-continuous, and consider a function $f:X \to \mathbb R$. For each $x\in X$, the following conditions are equivalent:
\begin{enumerate}
\item $f$ is asymptotically $p$-smooth at $x$
\item $\overline{\rho}_{f}(t;x)=O(t^p)$
\end{enumerate}
\end{prop}
\begin{proof} It is clear that $(2)$ implies $(1)$. Next we are going to prove the reverse implication. Suppose that   $f$ is asymptotically  $p$-smooth at $x$, and then there exist $K_x>0$,  $\delta_x>0$ and a polynomial $P_x$ of degree $<p$ on $X$,  with $P_x(0)=0$,  such that for each  $0< t < \delta_x$ there exists a finite codimensional subspace $H_{(t,x)}$ of $X$ verifying that if $h \in H_{(t,x)}$ and $\Vert h \Vert \leq t$ then
$$
|f(x+h)-f(x)-P_x(h)|\leq K_x \,\|h\|^p.
$$
Let  $P_x=P_1+\dots+P_k$, where $P_j$ is a $j$-homogeneous polynomial for  $j=1,\dots, k$. Consider a fixed $t$ with $0<t <\delta_x$.  By applying Lemma \ref{wb}  to each $P_j$, we obtain a finite codimensional subspace $Y=Y_{(t,x)}$ of $X$  such that  $\Vert P_j \Vert_{Y} < k^{-1} \, t^{p-j} $ for each $j=1, \cdots, k$. Now if  $h\in H$ with  $0<\Vert h \Vert \leq t$ we  have that
$$
\vert P(h)\vert \leq \vert P_1(h) \vert  + \dots + \vert P_k(h) \vert
= \Vert h \Vert \cdot \left| P_1\left(\frac{h}{\Vert h \Vert}\right) \right|  + \dots + \Vert h \Vert^k \cdot \left| P_k \left(\frac{h}{\Vert h \Vert}\right) \right| \leq
$$
$$
\leq t \frac{t^{p-1}}{k} + \dots + t^k \frac{t^{p-k}}{k} = t^p
$$
The subspace $Z_{(t,x)}= H_{(t,x)}\cap Y_{(t,x)}$ is finite codimensional and, for every  $h\in Z$ with $\Vert h \Vert \leq t$ we deduce that
$$
\vert f(x+h)-f(x) \vert \leq \vert f(x+h)-f(x)-P(h)\vert + \vert P(h) \vert
\leq K_x \,t^p + t^p = (K_x +1) t^p.
$$
Thus $\overline{\rho}_{f}(t;x)=O(t^p)$.
\end{proof}

The uniform version of the above result can be obtained along the same lines:

\begin{prop} \label{flatuniform}
Let $1<p <\infty$, let $X$ be a Banach space such that all polynomials of degree $< p$ are wb-continuous, and consider a function $f:X \to \mathbb R$. For each subset $S$ of $X$, the following conditions are equivalent:
\begin{enumerate}
\item $f$ is uniformly asymptotically $p$-smooth on $S$
\item $\overline{\rho}_{f}(t,S)=O(t^p)$
\end{enumerate}
\end{prop}

\section{Asymptotic smoothness of convex envelopes}

Now our result about asymptotic smoothness of convex envelopes will follow easily from the previous developments.

\begin{theorem}\label{envelope}
Let let $1< p < \infty$ and let  $X$ be a Banach space such that all polynomials of degree $< p$ are wb-continuous. Consider a function $f: X \to \mathbb{R}$ continuous and bounded below. If $f$ is uniformly asymptotically $p$-smooth on $X$, then also $\text{\sl conv}f$ is uniformly asymptotically $p$-smooth on $X$.
\end{theorem}

\begin{proof}
Since  $f$ is uniformly asymptotically $p$-smooth on $X$, by Proposition \ref{flatuniform} there exist  $K>0$ and $\delta >0$ such that
for each $t>0$ and $x \in X$ there is a finite codimensional subspace $H_{(t, x)}$ verifying that if $h\in H_{(t,x)}$ and $\Vert h \Vert \leq t < \delta$ then
$$
|f(x+h)-f(x)|\leq K \, t^p.
$$
Now fix $t>0$ and $x\in X$. By the definition of convex envelope, we can choose $x_{1}, \dots,x_{m} \in X $ and  $\lambda_{1}, \dots,\lambda_{m}\in \mathbb R$ with  $\lambda_i  \geq 0$ for each  $i=1, \dots, m$  and $\sum_{i=1}^{m} \lambda_i=1$, such that
$$
\sum_{i=1}^{m} \lambda_{i} f(x_{i}) \leq \text{\sl conv}f(x) + t^p.
$$
Then for each $i=1, \dots, m$ there is a finite codimensional subspace $H_{(t, x_i)}$ verifying that if $h\in H_{(t, x_i)}$ and $\Vert h \Vert \leq t < \delta$ we have
$$
|f(x_{i}+h)-f(x_{i})|\leq K \, t^p, i=1,\dots, m.
$$
Therefore if we consider the finite codimensional subspace $H =\cap_{i=1}^{m} H_{(t, x_i)}$ we have that for every  $h \in H$ with  $\Vert h \Vert \leq t < \delta$:
$$
\text{\sl conv}f(x+h) - \text{\sl conv}f(x) \leq
\sum_{i=1}^m \lambda_i f(x_i+h)- \sum_{i=1}^m \lambda_i f(x_i) + t^p
\leq K \, t^p  \sum_{i=1}^{n} \lambda_i + t^p
= (K + 1) t^p.
$$
Taking into account Lemma \ref{absolute} we obtain that $\overline{\rho}_{\text{\sl conv}f}(t,X)=O(t^p)$, and thus $\text{\sl conv}f$ is  uniformly asymptotically  $p$-smooth on $X$.
\end{proof}

\begin{remark} Let $P$ be a $N$-homogeneous polynomial on a Banach space $X$, and let $A: X \times \cdots \times X \to \mathbb R$ denote the associated symmetric $N$-linear form. From the binomial identity
$$
P(x+h) = \sum_{j=0}^{N} \binom{N}{j} A(x, \stackrel{N-j}{\dots}, \, x, h, \stackrel{j}{\dots}, \, h)
$$
it is plain that $P$ is uniformly $U^N$-smooth on $X$. Thus if all polynomials of degree $< N$ on $X$ are {\it wb}-continuous, we obtain that $\text{\sl conv}P$ is uniformly asymptotically  $N$-smooth on $X$.

Recall that a $N$-homogeneous polynomial on a Banach space $X$ is said to be {\it separating} (see e. g. \cite{Extracta} or \cite{Glez-Go}) if it separates the origin from the unit sphere, that is, there exists $\alpha >0$ such that $P(x) \geq \alpha$ for every $x\in X$ with $\Vert x \Vert=1$. Thus $P(x) \geq \alpha \Vert x \Vert^N$ for every $x\in X$, and $N$ has to be an even integer. It can be shown that in this case $\text{\sl conv}P$ is a positive, convex, $N$-homogeneous separating function. It follows from Theorem 3.5 in \cite{GJT} that, if $X$ admits a $N$-homogeneous separating polynomial and all polynomials of degree $< N$ on $X$ are {\it wb}-continuous, then $X$ admits an equivalent norm with modulus of asymptotic uniform smoothness of power type $N$.
\end{remark}

\section{Modulus of convexity of a polynomial norm.}

In this section we consider Banach spaces with a polynomial norm. Let $(X, \Vert \cdot \Vert)$ be a Banach space and  let $N$ be  an even integer. We say that $\Vert \cdot \Vert$ is a {\it polynomial norm} of degree $N$ on $X$ if there is an $N$-homogeneous polynomial $P$ on $X$ such that for every  $x\in X$,
$$
P(x)=\Vert x \Vert^N.
$$

\begin{remark}  Note that if there is a {\it convex}  $N$-homogeneous separating polynomial $P$ on $(X, \Vert \cdot \Vert)$, with $N$ being an even integer, then the space admits an equivalent polynomial norm. Indeed, the
expression:
$$
||| x ||| = \inf \Big{\{} \lambda >0 \, : \, P \left(\frac{x}{\lambda}\right)=1 \Big{\}}
$$
defines a norm on $X$ verifying that $P(x)=||| x |||^{N}$. Moreover, since  $P$ is a separating polynomial, there is $\alpha>0$ such that for
all $x\in X$,
$$
\alpha \Vert x \Vert^N \leq P(x)=||| x |||^{N} \leq \Vert P \cdot \Vert \Vert  x
\Vert^N
$$
\end{remark}
\noindent

\

The main result in this section shows that a polynomial  norm of degree $N$ has modulus of uniform convexity of power type $N$.  In order to obtain this, we will use the notion of $p$-uniformly convex norm introduced in \cite{Ball}  (see also \cite{Fiegel}) as follows. Let  $p>1$, for a Banach space  $(X, \Vert \cdot \Vert)$ the norm $\Vert \cdot \Vert$ is said to be $p$-{\it uniformly convex} if there exists a constant $K>0$ such that for every $x,y\in X$,
$$
2\Vert x\Vert^p + 2\frac{1}{K}\Vert y \Vert^p \leq \Vert x+y
\Vert^p
 + \Vert x-y \Vert^p.
$$
The best constant $K>0$ satisfying the above inequality is called the $p$-uniform convexity constant of the norm. In \cite{Ball} the authors  obtain
the following characterization of $p$-uniform convexity of a norm :

\begin{prop} \cite{Ball}  Let  $(X, \Vert \cdot \Vert)$ be a Banach space and $1<p< \infty$. The following are equivalent:

\begin{enumerate}
\item The norm $\Vert \cdot \Vert$ has modulus of uniform convexity of power type $p$.
\item The norm $\Vert \cdot \Vert$ is $p$-uniformly convex.
\end{enumerate}
\end{prop}

Next we recall the notion of  uniformly convex function introduced in \cite{BGHV}. Given a convex function $f:X \to \RR$  defined on a Banach space $X$, the {\it modulus of convexity of $f$}  is defined as the function $\delta_{f}:(0,\infty) \to [0,\infty)$ given by:
$$
\delta_{f}(t):= \inf {\Big{\{}} \frac{1}{2} f(x) + \frac{1}{2} f(y) - f(\frac{x+y}{2}) : \Vert x-y \Vert = t; \, x,y \in X {\Big{\}}}.
$$
The function $f$ is said to be {\it uniformly convex}  if $\delta_{f}(t)>0$ for all $t>0$. In \cite{BGHV} the authors  show that
the norm $\Vert \cdot \Vert$ has modulus of uniform convexity of power type $p$ if and only if the convex function $f(x)= \Vert x \Vert^p$ is
uniformly convex. This result is essential in order to prove our main  theorem:

\begin{theorem} Let $N$ be an even integer and let $(X, \Vert \cdot \Vert)$ be a Banach space such that $\Vert \cdot \Vert$  is a polynomial norm of degree $N$. Then $\Vert \cdot \Vert$ has modulus of uniform convexity of power type $N$.
\end{theorem}

\begin{proof} Let $P$ be be the convex $N$-homogeneous polynomial such that $P(x)= \Vert x \Vert^N$ for each $x\in X$. We will show that
$P:X \to \RR$ is $N$-uniformly convex and then the conclusion will follow from the above mentioned result in \cite{BGHV}. Let $t_0>0$ be fixed, and we have to show that
$$
\delta_{P}(t_0)=\inf {\Big{\{}} \frac{1}{2}P(x) + \frac{1}{2}P(y) -P(\frac{x+y}{2}) : \Vert x-y \Vert = t_0; \, x,y \in X {\Big{\}}} >0.
$$
By denoting $z=\frac{x+y}{2}$ and $h=\frac{x-y}{2t_0}$, this is equivalent to show that
$$
\inf {\Big{\{}}P(z+t_0h) + P(z-t_0h) -2P(z) : \Vert h \Vert = \frac{1}{2}; \, z \in X {\Big{\}}}>0.
$$
To this end, consider the set of polynomials in one variable:
$$
\mathcal{C}_{N}=\{p(t)=a_Nt^N + a_{N-2}t^{N-2} + \dots + a_{2}t^2 : p(t) \text{ is convex and } p(t) \geq 0 \text{ for all } t \in \mathbb R\}.
$$
Using the fact that $P$ is a convex function, it is easy to see that, for fixed $z, h\in X$ with $\Vert h \Vert = \frac{1}{2}$, the polynomial
$p_{z, h}(t) = P(z+th) + P(z-th) -2P(z) \in \mathcal{C}_{N}$. Now the set $\mathcal{C}_{N}$ is closed and convex in the space
$\mathcal{P}_N[t]$ of all real polynomials of degree  $\leq N$, when we consider, for instance, endowed with the norm given by
$$
\Vert p \Vert = \sum_{i=0}^{N} \vert a_i \vert  \qquad \text{for }  p(t)=a_0+a_1t + \dots + a_Nt^N.
$$
Indeed, convexity of $\mathcal{C}_N$ is clear. On the other hand,  assume that $\{p_n\}_{n=1}^{\infty}$ converges to $p$, with $p_n\in \mathcal{C}_{N}$ for each $n\in \NN$.  In particular, $p_{n}(t)\to p(t)$ for every $t\in \RR$.  Since the pointwise limit of non negative and  convex functions is again non negative and  convex, it follows that $p\in\mathcal{C}_N$ and thus $\mathcal{C}_{N}$ is a closed set.  Now, consider  the point
evaluation at $t_0 >0$ on the space $\mathcal{P}_N[t]$, that is, the linear functional $L_{t_0}: \mathcal{P}_N[t] \to \RR$ given  by $L_{t_0}(p)=p(t_0)$. This functional attains its minimum on $\mathcal{C}_N$  at a point,
say $q\in \mathcal{C}_N$. Consequently, for every $p\in \mathcal{C}_N$
$$
p(t_0) \geq q(t_0).
$$
Moreover, $q(t_0) >0$. Indeed, otherwise we would have  $q(0)=0=q(t_0)$ and since $q(t)$ is convex then  $q \equiv 0$, which is not possible. Therefore
$$
\inf \{ p(t_0) : p \in \mathcal{C}_N\} =q(t_0)>0.
$$
In particular, for every $z, h\in X$ with $\Vert h \Vert = \frac{1}{2}$ it follows that
$$
P(z+t_0h)+ P(z-t_0h)-2P(z) \geq q(t_0)>0,
$$
as we required.

\end{proof}

\begin{remark} It is interesting to point out that the proof of previous theorem  gives even more than the $N$-uniform convexity of a polynomial norm of degree $N$. In fact, it is proved that for  each even integer $N$ and $t_0>0$, there exists  a positive constant
$K(N,t_0)>0$ such that, for every polynomial norm $\Vert \cdot \Vert$ of degree $N$ in any Banach space $X$, if $x\in X$ and $\Vert h \Vert =t_0$  we have  that
$$
2 \Vert x \Vert^N + 2 \dfrac{1}{K(N,t_0)}\Vert h \Vert^N   \leq \Vert x+ h \Vert^N + \Vert x-h \Vert^N .
$$
One can wonder if the constant is uniform for every $t_0>0$, that is, if the $N$-uniform convexity constant is independent of the polynomial norm once
the degree is fixed. In this direction, for the cases of even integers $N$ with
$N\leq 6$ the answer is affirmative and this was already proved in \cite{Reznick}. We give an alternative proof of this result:
\end{remark}

\begin{prop} Let $(X, \Vert \cdot \Vert)$ be a Banach space with a polynomial norm
of degree $N$ where $N=2,4,6$. Then $\Vert \cdot \Vert$ has modulus of uniform  convexity of power type $N$, with constant  $1$.
\end{prop}

\begin{proof} In the case of $N=2$ the result is obvious since the norm is Hibertian, and consequently has modulus of uniform
convexity of power type $2$ with constant $1$. Moreover for every $x,h\in X$ we have:
$$
\Vert x +h \Vert^2 + \Vert x-h \Vert^2=  2\Vert x \Vert^2 + 2\Vert h \Vert^2.
$$
Let now consider the case $p=4$. Assume that $\Vert x \Vert^4 =P(x)=A(x,x,x,x)$ where  $A$ is the $4$-linear symmetric form
associated to  $P$. Note that since $P$ is convex, its second derivative satisfies
$D^2 P(x)(h)\geq 0$, and consequently, $A(x,x,h,h) \geq 0$, for every $x,h \in X$. Then,
$$
\Vert x +h \Vert^4 + \Vert x-h \Vert^4 = A(x+h,x+h,x+h,x+h) +
A(x-h,x-h,x-h,x-h)=
$$
$$
= 2\Vert x \Vert^4 + 2\Vert h \Vert^4+ 12A(x,x,h,h) \geq 2\Vert x
\Vert^4 + 2\Vert h \Vert^4.
$$
Thus  the norm has modulus of uniform convexity of power type $4$ with $4$-convexity constant
equals to $1$. The proof for $p=6$ is analogous; note that if
$\Vert x \Vert^6=P(x)= A(x,x,x,x,x,x)$ then by the convexity of
the polynomial $P$ we have $A(x,x,h,h,h,h)= A(h,h,h,h,x,x) \geq
0$ and the conclusion follows in the same way.
\end{proof}

We finish the paper by giving power type estimates for  the modulus of  asymptotic uniform
convexity and smoothness of a polynomial norm of degree $N$,  when every polynomial of degree $<N$ is {\it wb}-continuous.
We will prove the existence of universal constants for all polynomial norms of a fixed degree $N$, independent on the Banach space.

\begin{prop}  Let $N$ be an even integer. Then, there exist constants $c_{N}, C_{N} >0$ such that every Banach space $(X, \Vert \cdot \Vert)$ endowed with a polynomial norm of degree $N$, and such that every polynomial of degree $<N$ is {\it wb}-continuous, has moduli of  asymptotic uniform convexity and  smoothness of power type $N$, verifying that for every $0<t \leq 1$,
$$
\overline{\rho}(t) \leq c_{N}t^{N},  \qquad \qquad \overline{\delta}(t) \geq C_{N}t^{N}.
$$

\end{prop}

\begin{proof} Let $N$ be an even integer and let $(X, \Vert \cdot \Vert)$ be a Banach space such that every polynomial of degree $<N$ is {\it wb}-continuous, and the norm is given by $P(x)=\Vert x \Vert^N$, where $P$ is an $N$-homogeneous polynomial on $X$. Let $x\in X$ with $\Vert x \Vert =1$ and $0<t\leq 1$ be fixed. Now consider $A$  the continuous,  $N$-linear symmetric form associated to $P$. Thus if $h\in X$ with $\Vert h
\Vert=1$, using the binomial formula we have that
$$
\Vert x+th\Vert^{N} =P(x+th) = A(x+th,\dots,x+th)= 1+t^N
+
\sum_{i=1}^{N-1} t^i P_{i, x}(h)
$$
where $P_{i,x}$ is an $i$-homogeneous polynomial for $i=1,\dots, N-1$, depending on $x$.

By Lemma \ref{wb} there is a finite codimensional subspace $H_{(t,x)}$ of $X$ such that:
$$
\max_{1\leq i \leq N-1} \Vert P_i \Vert_{H_{t,x}}
<\frac{t^{N-1}}{2N}.
$$
Then,
$$
\left| \sum_{i=1}^{N-1} t^i P_{i}(h) \right| \leq
tN \max_{1\leq i \leq N-1} \Vert P_i \Vert_{H} <\frac{t^N}{2}.
$$
Therefore, for every $h\in H_{(t,x)}$ with $\Vert h \Vert =1$, we have that
$$
\Vert x+th \Vert^N -1  \geq t^N- \left| \sum_{i=1}^{N-1} t^i P_{i}(h) \right|
\geq t^N-\frac{t^N}{2}=\frac{1}{2}t^N.
$$
and
$$
\Vert x+th \Vert^N -1\leq t^N+ \left| \sum_{i=1}^{N-1}\vert t^iP_i(h)\right|
\leq \frac{3}{2}t^N.
$$
Note in particular that $1\leq \Vert x+th \Vert \leq 2$ in this case. Then as an easy consequence of the mean value theorem applied to the
function $\lambda(u)=u^{1/N}$ on the interval $I=[1,2]$,  there
exist constants only depending on $N$, say  $\alpha_N , \beta_N>0$ such that if $h\in H_{(t,x)}$ and $\Vert h
\Vert=1$ we have
$$
\frac{1}{2}\alpha_N t^N \leq \alpha_N(\Vert x+th\Vert^N-1)  \leq \Vert x +
th\Vert -1 \leq \beta_N (\Vert x+th\Vert^N-1) \leq \frac{3}{2} \beta_{N} t^N.
$$
Now by using Lemma 2.1 in \cite{GJT} we have that, for every  $x\in X$ with $\Vert x \Vert =1$ and every $0<t \leq 1$,
$$
\overline{\delta} (t;x)= \sup_{dim(X/H)<\infty} \qquad \inf_{h\in H,
\,\, \Vert h \Vert = 1} \qquad \Vert x+ th \Vert -1 \geq c_{N}t^N,
$$
and
$$
\overline{\rho}(t;x)=\inf_{dim(X/H)<\infty} \qquad \sup_{h\in H,
\,\, \Vert h \Vert = 1}  \qquad\Vert x+th \Vert -1 \leq C_{N}t^{N},
$$
for the constants  $c_{N}= \frac{1}{2} \alpha_N$ and $C_{N}= \frac{3}{2}\beta_N$ depending only on $N$.
Consequently, for every $0<t \leq 1$
$$
\overline{\rho}(t) \leq c_{N}t^{N} \qquad \qquad \overline{\delta}(t) \geq C_{N}t^{N}.
$$

\end{proof}

\begin{remark} Note that in the above proposition the
hypothesis of {\it wb}-continuity of polynomials of degree  $<N$ cannot be removed, since for instance the
space $\ell_2\bigoplus \ell_4$ admits a polynomial norm of degree
$4$ and nevertheless it has modulus of asymptotic uniform
smoothness of exact power type $2$ and modulus of asymptotic
uniform convexity of exact power type $4$.
\end{remark}

\end{document}